\documentclass{article}

\usepackage{arxiv}
\usepackage{amsmath}
\usepackage{enumitem}
\usepackage{amssymb}
\usepackage{graphicx}
\usepackage{booktabs}
\usepackage[utf8]{inputenc} 
\usepackage[T1]{fontenc}    
\usepackage{hyperref}       
\usepackage{url}            
\usepackage{booktabs}       
\usepackage{amsfonts}       
\usepackage{nicefrac}       
\usepackage{microtype}      
\usepackage{lipsum}
\usepackage{amsthm}
\newtheorem{theorem}{Theorem}

\newtheorem{proposition}{Proposition}
\newtheorem{corollary}{Corollary}
\newtheorem{problem}{Problem}

\theoremstyle{definition}
\newtheorem{definition}{Definition}

\theoremstyle{remark}
\newtheorem{remark}{Remark}

\title{Temperature control by a state-delay Nash strategy: theory and experiments}

\author{
  C\'esar--Adair Rodr\'iguez--Castro  \\
    Academic Area of Computing and Electronics,\\
    Autonomous University of the State of Hidalgo,\\
    Pachuca, Hidalgo 42184, Mexico\\ \texttt{cesar\_rodriguez@uaeh.edu.mx} \\
   \And
 Manuel--Alejando J\'imenez--Lizarraga \\
  Faculty of Physical and Mathematical Sciences\\
Autonomous University of Nuevo León\\
Monterrey, Nuevo León 66455, Mexico\\  \texttt{manuel.jimenezlzr@uanl.edu.mx} \\
  \And
 Jorge--Manuel Ortega--Mart\'inez \\
  Department of Electrical Engineering\\
  Technological Institute of Morelia\\
  Morelia, Michoacán 58120, Mexico\\
  \texttt{jorge.om@morelia.tecnm.mx} \\
  \And
 Omar--Jacobo Santos--S\'anchez     \thanks{Corresponding author.}\\
    Academic Area of Computing and Electronics,\\
    Autonomous University of the State of Hidalgo,\\
    Pachuca, Hidalgo 42184, Mexico\\
  \texttt{omarj@uaeh.edu.mx} \\
}

\begin{document}
\maketitle

\begin{abstract}
This work expands the Nash equilibrium algorithm to those systems that have a state delay within the deterministic case, a Linear-Quadratic (LQ) type performance index with generalized cross terms is used, following the constructive approach on optimal control: first, propose the strategies form, close the loop and find the Cauchy solution; secondly, express the cost function according to those expressions; and finally use the Bellman equation as restrictions to find the strategies parameters.
    The stability of the closed loop is demonstrated. 
    The algorithm is applied to a thermal prototype with an ESP32 micro-controller and its performance is compared with optimal PI controls inside two commercial/industrial PID controllers REX-C100.
\end{abstract}

\keywords{Delay systems \and ESP32 \and Nash Equilibrium \and Optimal control \and Temperature control}

\section{Introduction}
For feedback strategies in dynamic games \cite{bacsar1998dynamic,dockner2000differential,engwerda2005lq}, the presence of delays \cite{kharitonov2012time,kolmanovskii2012applied,fridman2014introduction} requires reformulating the problem within a functional state space that accounts for all the state of time delay systems. This approach naturally leads to a set of Riccati-type equations, whose explicit solution is not straightforward \cite{curtain2012introduction,bensoussan2007representation,kim2000linear}. As a consequence, the majority results, of the synthesis of optimal control for time delay systems available, in the literature are expressed in implicit terms or rely on numerical methods \cite{krasovskii1962analytic,ross1969optimal,fridman2014introduction}.

Although recent progress has been made of Linear-Quadratic (LQ) time delay-games approach, this mainly is in stochastic, mean-field or hierarchical settings, and discrete time setups \cite{kong2022leader,oliveira2020nash,mukaidani2013dynamic,zhao2022varepsilon,oksendal2011optimal,huang2018linear,meng2020linear}. The deterministic case with state delays has received comparatively less attention \cite{tan2017delay,dong2015formation}. In particular, the analytical characterization of the parameters of the Nash equilibrium in deterministic LQ delay differential games remains, to a large extent, an open problem.

These difficulties motivated this work, which studies a deterministic two-player LQ-type differential game in which the state dynamics includes a delay term. The main objective is to analyze how the presence of the delayed state in the system, affects the synthesis of the Nash equilibrium and to derive a set of Riccati-type equations that enable its analytic solution. It allows for explicit formulas for each feedback Nash strategy. 

Despite the presence of state delay in the dynamics in the mathematical model, the particular structure of the performance index (cross terms, originally proposed in \cite{kim2000linear}) allows an analytic solution for the equilibrium, by obtaining the solution of a matrix set of Partial Differential Riccati-type Equations. In \cite{kim2000linear} a time delay control system is approached, in this contribution this approach is extended to the time delay dynamic games.

The exposed algorithm is validated through its implementation in a prototype thermal plant, with two actuators (its players) demonstrating its feasibility and good performance under real operating conditions; closed-loop Nash strategies are deployed to the ESP32 hardware. This implementation allows to assess the effectiveness of the strategies by a comparison with an optimal PI that has been carried on. The optimal PI controllers are synthesized by using the algorithm presented in \cite{he2000pi} and deployed on two REX-C100 controllers.
The following points summarize the contributions of this proposal:
\begin{enumerate}
     \item In contrast to the approach presented in \cite{kim2000linear}, the Bellman functional presented here is constructed for performance indexes that include cross terms.
    \item The extension of the linear optimal control approach with time-delays with cross terms in the performance index \cite{kim2000linear}, to dynamic games with time-delays, without using the i-smooth calculus.
    \item Explicit formulas are provided for each Nash feedback strategy, which facilitates its real-time implementation.
    \item Experimental validation is carried out on thermal systems with resistive heaters, where two heaters represent the players. The proposed closed loop system exhibits improved transient performance when it is compared with two optimally tuned commercial PI controllers.
\end{enumerate}
\textbf{Notation.} 
In this case, the subscripts $j$ and $i$ denote the player, the subscript $-i$ refers to all other players, while the superscript $-1$ denotes the inverse matrix and $T$ denotes the matrix transpose. The space $PC\left[-h,0\right]$ denotes the set of the piecewise continuous functions defined over the interval $\left[-h,0\right]$. $|\cdot|$ is the euclidean norm, $\|x_t\|_h$ describes $\max_{\theta\in[-h,0]}|\varphi(\theta)|$. \\

The next section is devoted to presenting the fundamental results required to address the problem considered in this work.

\section{Preliminaries}\label{prelim}
	
	In this section, topics related to understanding the development of this work are presented.
		
	\subsection{Lyapunov stability for time-delay systems}
	Let the system be
	\begin{align}
		\begin{cases}
			\dot x(t)= f(t,x_t), & t\geq t_0, \quad f(t,0)\equiv 0;\\
			x_t : =x(t+\theta), & \theta \in[-h,0]; \quad x_{t_0}=\varphi.
		\end{cases}
		\label{sr1}
	\end{align}
	where $h>0$, additionally, it is assumed that
	 $\varphi\in PC[-h,0]$ and that the functional $((t,\varphi)\to f(t,\varphi) : \mathbb{R}\times Q_H \to \mathbb{R}^n)$ is continuous, bounded, and locally Lipschitz in $\varphi\in Q_H:=\{\varphi\in PC[-h,0]:\|\varphi\|<H\}$ for some $H>0$.
	
	The following concepts define and provide sufficient conditions for the stability, in the Lyapunov sense, of the system given by \eqref{sr1}.
	
	\begin{definition}[Decreasing functional]\cite{kolmanovskii2012applied}
		Let $\Omega,\ \omega\in C[0,\infty)$ be the set of increasing scalar functions such that $\omega(r)>0$ for $r>0$ and $\omega(0)=0.$ A functional, $(t,u)\to V(t,u)$, it is said to be positive definite if there exists a function
		 $\omega\in\Omega$ such that $$V(t,\varphi)\geq(\|\varphi(0)\|),$$ for all $t\in\mathbb{R}, \ \varphi\in Q_H$
	\end{definition} 
	
	\begin{theorem} \label{Stability}
	\cite{fridman2014introduction}
		Suppose $f:\mathbb{R}\times C[-h,0]\to\mathbb{R}^n$ maps $\mathbb{R}\times$(bounded sets in $C[-h,0]$) into bounded sets of $\mathbb{R}^n$ and that $u,\ v,\ w:\mathbb{R}_+\to\mathbb{R}_+$ are continuous non-decreasing functions, $u(s)$ and $v(s)$ are positive for $s>0$, and $u(0)=v(0)=0$. The trivial solution of (\ref{sr1}) is uniformly stable if there exists a continuous functional $V:\mathbb{R}\times C[-h,0]\to\mathbb{R}^+$, which is positive definite \begin{align*}
		    u(|\varphi(0)|) \leq V(t,\varphi) \leq v\left( \|x_t\|_h \right)
		\end{align*}
        such that its derivative along (\ref{sr1}) is non positive in the sense that $\dot V(t,\varphi)\leq-w(|\varphi(0)|)$.

        If $w(s)>0$ for $s>0$ then the trivial solution is uniformly asymptotically stable. If in addition $\lim_{s\to\infty}u(s)=\infty$, then it is globally uniformly asymptotically stable

	\end{theorem}

	\subsection{Dynamic programming for a time delay optimal control}
	The sufficient conditions for an optimal control of a system with delays, according to Dynamic Programming are:
	
	\begin{theorem}\label{Bellman}
		\cite{ross1969optimal,ortega2021comments}, If there is an admissible $u^*(t)=u^*(x_t)$ and a continuous scalar non negative functional $V(x_t)=0 \ \forall\  x_t=0$, such that 
		\begin{align*}
			\dot V(x_t)\big|_{_{\substack{(\ref{S1}) \\u=u^{\ast}
			}}}+g(x_t,u^*(x_t))=0,									\\
			\dot V(x_t)\big|_{_{\substack{(\ref{S1}) \\u=u^{\ast}
			}}}+g(x_t,u^*(x_t))\leq\dot V(x_t)\big|_{_{\substack{(\ref{S1}) \\u=u(t)
			}}}+g(x_t,u(t))	,										
		\end{align*}\normalsize
		for all $ t>0 $ and any admissible control, then $u^*$ is an optimal control. Furthermore $V(\phi)=J(\phi,u^*)$ is the optimal value of $J$, the performance index.  
	\end{theorem}
	
	\subsection{Free delay Nash equilibrium}
	Consider the minimization of the performance index
	\begin{align*}
		J_i(\cdot)=\lim_{T\to\infty} J_i (x_0,u_1,u_2,T),
	\end{align*} 
	for each player $i,\ i=1,2$, where
	
	\begin{align*}
		J_i&(\cdot)=\lim_{T\to\infty}\int_0^T\left\{ x^t(t)Q_ix(t) + u_i^T(t)R_{i,j}u_i(t)\right. \\ 
        &+ \left. u_i^T(t)R_{i,i}u_i(t)\right\}dt, \quad j\neq i ,
	\end{align*}
	  subject to
	
	\begin{align*}
		\dot x (t) = A(x) + B_1u_1(t) + B_2u_2(t),\qquad x(0)=x_0,
	\end{align*}
	here $Q_i,R_{i,j}$ are symmetric and $R_{i,i}$ is positive definite.
	
	The admissible controls are constant linear feedback strategies, that is,
	$u_i=F_ix,\quad F_i\in\mathbb{R}^{m_i\times n},\quad i=i,2$, where $(F_1,F_2)$ belong to
	\begin{align*}
		\mathcal{F}:=\left\{ F=(F_1,F_2)|A+B_1F_1+B_2F_2 \text{ is stable} \right\}.
	\end{align*}
	\begin{definition}\cite{engwerda2005lq}
		$ (F_1^*,F_2^*) \in \mathcal{F} $ is called a stationary linear feedback Nash equilibrium if the following inequalities are satisfied: 
			$$J_1(x_0,F_1^*,F_2^*)\leq J_1(x_0,F_1,F_2^*),\quad J_2(x_0,F_1^*,F_2^*)\leq J_1(x_0,F_1^*,F_2),$$
		
		for each $x_0$ and matrix state feedback $F_i,\ i=1,2$, such that $ (F_1^*,F_2),(F_1,F_2^*) \in \mathcal{F} .$
	\end{definition}

    \begin{theorem}\cite{engwerda2005lq}\label{Nash non delay}
        Let $(K_1,K_2)$ be a symmetric stabilizing solution of 
        
        \begin{align*}
            0=&-(A-S_2K_2)^TK_1-K_1(A-S_2K_2)\\
            &+K_1S_1K_1-Q_1-K_2S_{2,1}K_2,\\
            0=&-(A-S_1K_1)^TK_2-K_2(A-S_1K_1)\\
            &+K_2S_2K_2-Q_2-K_1S_{1,2}K_1,
        \end{align*}
          and define $F^*_i:=-R_{i,i}^{-1}B_i^TK_i $ for $i=1,2$. Then $(F_1^*,F_2^*)$ is a feedback Nash equilibrium.
    \end{theorem}
	
	Once the previous results have been recalled, it is possible to establish the main problem addressed in this proposal.

\section{Problem statement}
Let the system be described by 
	\begin{equation} \label{S1}
		\dot{x}=A_{0}x\left(  t\right)  +A_{1}x\left(  t-h\right)  +B_{1}u_{1}\left(
		t\right)  +B_{2}u_{2}\left(  t\right)  ,
	\end{equation} 
	where $A_{0},A_{1} \in\mathbb{R}^{n\times n},\ B_{i}  \in\mathbb{R}^{n\times m_{i}},m_{i}\leq n,	x\left(  t\right)   \in\mathbb{R}^{n},$
	$u_{i}\left(  t\right)\in	\mathbb{R}^{m_{i}},  \ i=1,2,\
	$ whose initial conditions are $x\left(  \theta\right)  =\varphi\left(  \theta\right),\ \theta\in [-h,0]  ,\  \varphi\in PC\left[-h,0\right]$ and the state is defined as $x_{t}=x\left(  t+\theta\right)$, where $\theta\in [-h,0]$.
	Now, the performance indices for each player are defined as follows:

	\begin{align}			\label{J_1}						
		J_{i}\left(  u_{i},u_{-i}\right)   &  =\int_{0}^{\infty}\left\{  x^{T}\left(
		t\right)  \Phi_{i,0}x\left(  t\right) \right.\\
        + &  \left.  2x^{T}\left(  t\right)  \int
		_{-h}^{0}\Phi_{i,1}\left(  \theta\right)  x\left(  t+\theta\right)  d\theta 				\right.  \nonumber \\
		+  &  \left.  \int_{-h} ^{0}\int_{-h}^{0}x^{T}\left(  t+\xi\right)  \Phi_{i,2}%
		\left(  \xi,\theta\right)  x\left(  t+\theta\right)  d\xi d\theta								\right.	\nonumber\\
		+ & \left. u_i^{T}(t)  R_{i,i} u_{i}\left(  t\right) + u_{-i}^{T}\left(
		t\right)  R_{i,-i} u_{-i}\left(  t\right) \right\}
		dt.\nonumber
	\end{align}\normalsize
    
	Here, $\Phi_{i,0}$ are $n\times n$ symmetric and positive definite matrices, $\Phi_{i,1}\left(  \theta\right)$ are $n\times n$ continuous matrix functions in $\left[  -h,0\right]$ and $\Phi_{i,2}\left(  \xi,\theta\right)$ are $n\times
	n$ continuous matrix functions in $\xi\in\left[  -h,0\right]$ and $\theta\in\left[  -h,0\right]$, also $\Phi_{i,2}^{T}\left(  \xi
	,\theta\right)  =\Phi_{i,2}\left(  \theta,\xi\right)$, where $i=1,2$.
	Now, consider the following time-delay differential game problem, which represents the main issue addressed in this work.
	\bigskip
	\begin{problem}
	The Nash equilibrium for the two-player system with delay can be formulated as an optimal control problem as follows: given the system (\ref{S1}) and the performance indices (\ref{J_1}), the goal is to synthesize the strategies \(u_1^*(x_t)\) and \(u_2^*(x_t)\) such that the performance indices (\ref{J_1}) are minimized, subject to the system trajectories (\ref{S1}).
	\end{problem}

	The admissible strategies for this problem must satisfy:
	\begin{enumerate}[label=(\roman*)]
		\item the control strategies are functions that depend on the state: $u_i(t)=u_i(x_t)$,
		\item $u_i(x_t)$ are such that the solution of (\ref{S1}) exists and is unique for all $t\geq0$ and any $\varphi$,
		\item the closed-loop trivial solution of the system (\ref{S1}) with $u_i(x_t)$ is asymptotically stable,
		\item $u_i(x_t)$ is continuous for all $t\geq 0$.
	\end{enumerate}

	\begin{remark}
		When the performance index is quadratic, it is possible to construct the Bellman functional following the procedure presented in \cite{ortega2021comments}, obtaining Riccati-type equations similar to those given in \cite{kim2000linear,ross1969optimal} for the multi-player case. Moreover, it is possible to derive the explicit form of the control strategies. However, solving the Riccati-type equations is not a simple task. In fact, this difficulty arises from the coupling among the unknown matrices in the set of Riccati-type equations (see Appendix).
	\end{remark}

	The reasons outlined above motivate the use of the approach proposed by \cite{kim2000linear}, which relies on the idea of decoupling Riccati-type equations to obtain an analytical solution to the resulting set of equations. However, in this contribution, the use of i-smooth calculus \cite{kim2015smooth} is avoided. Nevertheless, the Bellman functional is not positive definite so, similarly to exposed in \cite{kim2015smooth},  the stability is addressed separately.
	\section{Main theoretical results} \label{Resultados teoricos}
	The methodology is related to that of Theorem \ref{Bellman} to find the $u_i^{*}(x_t)$ and the functionals $V_i(x_t)$ for system (\ref{S1}) such that they satisfy the conditions:
	 \begin{align}
		\dot V_i(x_t)&\big|_{_{\substack{(\ref{S1}) \\u_i=u_i^{\ast}
		}}}+g_i(x_t,u_i^*(x_t),u_{-i}^*(x_t))=0,										\label{v_c1}\\
		\dot V_i(x_t)&\big|_{_{\substack{(\ref{S1}) \\u_i=u_i^{\ast}
		}}}+g_i(x_t,u_i^*(x_t),u_{-i}^*(x_t))\leq  
        \dot V_i(x_t)\big|_{_{\substack{(\ref{S1}) \\u_i=u_i
		}}}+g_i(x_t,u_i(t),u_{-i}^*(x_t))	,				\label{v_c2}
	\end{align}
	that can be archived by the next three steps:  
	\begin{enumerate}
		\item Choose the particular form of the strategies $u_i^{*}(x_t)$, then close the loop and find the Cauchy solution of the system.	\label{paso1}
		\item The performance index of each player is expressed, with this strategies form and solution, as an explicit function of the initial state:
		\begin{equation*}\label{paso2}
			V_i(\varphi) = J_{i}(\varphi,u_i^{*}(x_t),u_{-i}^*(x_t)),
		\end{equation*}						
		in contrast to Theorem \ref{Bellman}, it is difficult to determine that this performance index is non negative by the cross terms.
		\item Equations (\ref{v_c1}) and (\ref{v_c2}) are used as restrictions on $u_i^{*}(x_t)$ parameters. 		\label{paso3}
	\end{enumerate}

	The admissible strategies $u_i(x_t)$ structure, since they depend on the entire state, according to \cite{krasovskii1962analytic,ortega2021comments}, is:
	
	\begin{equation}
		u_{i}\left(  x_t\right)  =\Gamma_{i,0}x\left(  t\right)  +\int_{-h}^{0}%
		\Gamma_{i,1}\left(  \theta\right)  x\left(  t+\theta\right)  d\theta.
		\label{u_i}%
	\end{equation} \normalsize
	where the parameters are $\Gamma_{i,0} ,\Gamma
	_{i,1}\left(  \theta\right)  \in\mathbb{R}^{m_{i}\times n},\ i=1,2;$
	when the loop is closed with the control strategies (\ref{u_i}) in system (\ref{S1}), one obtains
	
	\begin{align*}
		\dot{x}=&A_{0}x\left(  t\right)  +A_{1}x\left(  t-h\right) \\
        & + B_{1}\left(  \Gamma_{1,0}x\left(  t\right)  +\int_{-h}^{0}\Gamma_{1,1}%
		\left(  \theta\right)  x\left(  t+\theta\right)  d\theta\right)  				\nonumber\\
		& +	B_{2}\left(  \Gamma_{2,0}x\left(  t\right)  +\int_{-h}^{0}\Gamma_{2,1}%
		\left(  \theta\right)  x\left(  t+\theta\right)  d\theta\right) . \nonumber
	\end{align*}\normalsize

	Define:	$		\tilde{A}_{0}=A_{0}+
		B_{1}\Gamma_{1,0} +	B_{2}\Gamma_{2,0} ,\quad	   
		G\left(  \theta\right)  =
		B_{1}\Gamma_{1,1}\left(  \theta\right)+B_{2}\Gamma_{2,1}\left(  \theta\right),				
	$
	then, this implies that the closed-loop system of (\ref{S1}) with the control strategies (\ref{u_i}) has the form
	
	\begin{align}
		\dot{x}\left(  t\right)  =\tilde{A}_{0}x\left(  t\right)  +A_{1}x\left(
		t-h\right)  +\int_{-h}^{0}G\left(  \theta\right)  x\left(  t+\theta\right)
		d\theta,																		\label{S1_lc}
	\end{align}\normalsize
	whose solution in Cauchy form  \cite{kolmanovskii2012applied} is
    
	\begin{align}
		x\left(  t,\varphi\right)  &=K\left(  t\right)  \varphi\left(  0\right)
		+\int_{-h}^{0}K\left(  t-\theta-h\right)  A_{1}\varphi\left(  \theta\right)
		d\theta                                                                     \label{Sol_Cauchy}\\
		&\ \ +\int_{-h}^{0}\int_{-h}^{\theta}K\left(  t-\theta+\xi\right)  G\left(
		\xi\right)  d\xi\varphi\left(  \theta\right)  d\theta,						\nonumber
	\end{align}\normalsize
	here $K(t)$ is the system (\ref{Sol_Cauchy}) fundamental matrix, this implies $K(0) = I_{n \times n}$ and $K(\tau) = 0_{n \times n}$ for all $\tau < 0$.

	The construction of the Bellman functional from the performance indices is presented below:
	
	\begin{proposition}														\label{prop1}
		If $ u_i(x_t) $ are linear admissible control strategies that satisfy a Lipschitz condition for all $t\leq0$ and $\varphi$ an arbitrary initial condition, then 
		
		\begin{align}			\label{J_i}								
			J_{i}&\left( x_t, u_{i},u_{-i}\right)     =\int_{0}^{\infty}\left\{  x^{T}\left(
			t\right)  \Phi_{i,0}x\left(  t\right)  \right. \\
            & +\left.2x^{T}\left(  t\right)  \int
			_{-h}^{0}\Phi_{i,1}\left(  s\right)  x\left(  t+\theta\right)  d\theta 	\right. 			\nonumber\\
			&+    \left.  \int_{-h} ^{0}\int_{-h}^{0}x^{T}\left( t+\xi \right)  \Phi_{i,2}%
			\left(  \xi,\theta\right)  x\left( t+\theta\right)  d\xi d\theta \right.\nonumber\\
			&+ \left.   u_1^{T}(t)  R_{i,1} u_{1}\left(  t\right) + u_2^{T}\left(
			t\right)  R_{i,2} u_{2}\left(  t\right) \right\}
			dt																		\nonumber	
		\end{align}\normalsize
		can be represented in the form
        
		\begin{align}\label{V_p}
			V_{i}\left(  \varphi\right)  =&\varphi^{T}\left(  0\right)  \Pi_{i,0}%
			\varphi\left(  0\right)  +2\varphi^{T}\left(  0\right)  \int_{-h}^{0}\Pi_{i,1}\left(  \theta\right)  \varphi\left(  \theta\right)  d\theta  \\
			&+\int
			_{-h}^{0}\int_{-h}^{0}\varphi^{T}\left(  \xi\right)  \Pi_{i,2}\left(
			\xi,\theta\right)  \varphi\left(  \theta\right)  d\xi d\theta.					\nonumber
		\end{align}\normalsize
		where the matrices $\Pi_{i,0}$, $\Pi_{i,1}\left(  \theta\right)$ and $\Pi_{i,2}\left(  \xi,\theta\right)$ satisfy the conditions:		
		\begin{enumerate}
			\item $\Pi_{i,0}$ is a $n\times n$ positive definite ans symmetric matrix,
			\item $\Pi_{i,1}\left(  \theta\right)$ is a $n\times n$ of continuous functions, in $\left[  -h,0\right]$, matrix; and 
			\item $\Pi_{i,2}\left(  \xi,\theta\right)$ is a $n\times
			n$ of continuous functions in $\xi\in\left[  -h,0\right]$ and $\theta\in\left[  -h,0\right]$ matrix; and $\Pi_{i,2}^{T}\left(  \xi
			,\theta\right)  =\Pi_{i,2}\left(  \theta,\xi\right)$.
		\end{enumerate}
	\end{proposition}

	\begin{proof}
		First, consider the system (\ref{S1}) and the control strategies (\ref{u_i}), then, the closed-loop system (\ref{S1})-(\ref{u_i}) is given by (\ref{S1_lc}). Now, the admissible strategies (\ref{u_i}) is substituted in the performance index (\ref{J_i}), it follows that
        
		\begin{align}
			J_{i}\left(  u_{i},u_{-i}\right)   = &\frac{1}{2}\int_{0}^{\infty}\left(
			x^{T}\left(  t\right)  M_{i,1}x\left(  t\right) \right.\label{J_m} \\
            &+2x^{T}\left(  t\right)
			\int_{-h}^{0}M_{i,2}\left(  \theta\right)  x\left(  t+\theta\right)	d\theta 			\nonumber\\
			&\left.+\int_{-h}^{0}\int_{-h}^{0}x^{T}\left(  t+\xi\right)  M_{i,3}%
			\left(  \xi,\theta\right)  x\left(  t+\theta\right)
			d\xi d\theta\right)  dt .			\nonumber
		\end{align}\normalsize
		with 
		\begin{align}
			M_{i,1}=& \Phi_{i,0}+
			\Gamma_{1,0}^{T}R_{i,1}\Gamma_{1,0}
			+ \Gamma_{2,0}^{T}R_{i,2}\Gamma_{2,0}, 										\label{Ms1}\\
			M_{i,2}\left(  \theta\right)  =& \Phi_{i,1}(\theta)+
			\Gamma_{1,0}^{T}R_{i,1}\Gamma_{1,1}\left(  \theta\right) 
			+ \Gamma_{2,0}^{T}R_{i,2}\Gamma_{2,1}\left(  \theta\right),					 \label{Ms2}\\
			M_{i,3}\left(  \xi,\theta\right)  =& \Phi_{i,2}(\xi,\theta)+
			\Gamma_{1,1}^{T}\left(  \xi\right)  R_{i,1}\Gamma_{1,1}\left(
			\theta\right) \label{Ms3}\\ &+
			\Gamma_{2,1}^{T}\left(  \xi\right)  R_{i,2}\Gamma_{2,1}\left(
			\theta\right). 			\nonumber												
		\end{align}\normalsize

		Then, the Cauchy-form solution (\ref{Sol_Cauchy}) can be rewritten as 
		\begin{equation}
			x\left(  t,\varphi\right)  =  K\left(  t\right)  \varphi\left(
			0\right)  +\int_{-h}^{0}\tilde K\left(  t,\theta\right)  \varphi\left(
			\theta\right)  d\theta,	\label{S_C}
		\end{equation} \normalsize
		where 
		\begin{align*}
			\tilde{K}(t,\theta)  & =K\left(  t-\theta-h\right)  A_{1}+\int_{-h}^{\theta}K\left(  t-\theta+\xi\right)  G\left( \xi\right)  d\xi		,	 
		\end{align*}\normalsize
		finally,  substituting (\ref{S_C}) in (\ref{J_m}), we obtain: 
        
		\begin{align}
			J_{i}\left(  u_{i},u_{-i}\right)  & =\varphi^{T}\left(  0\right)  \Pi_{i,0}%
			\varphi\left(  0\right)  +2\varphi^{T}\left(  0\right)  \int_{-h}^{0}%
			\Pi_{i,1}\left(  \theta\right)  \varphi\left(  \theta\right)  d\theta  			\nonumber \\ 
			+&\int_{-h}^{0}\int_{-h}^{0}\varphi^{T}\left(  \xi\right)  \Pi_{i,2}\left(
			\xi,\theta\right)  \varphi\left(  \theta\right)  d\xi d\theta; 						\label{J}
		\end{align}\normalsize
		with (\ref{Pi_0})--(\ref{Pi_2}).

		\tiny
		\begin{align}\label{Pi_0}
			\Pi_{i,0}  =\frac{1}{2}\left(  \int_{0}^{\infty}K^{T}\left(  t\right)
			M_{i,1}K\left(  t\right)  dt\right)
			+\int_{-h}^{0}\left(  \int_{0}^{\infty}K^{T}\left(  t\right)  M_{i,2}%
			\left(  \xi\right)  K\left(  t+\xi\right)  dtd\xi\right)											
			+\frac{1}{2}\left(  \int_{-h}^{0}\int_{-h}^{0}\int_{0}^{\infty}K^{T}\left(
			t+\xi\right)  M_{i,3}\left(  \xi,\theta\right)  K\left(
			t+\theta\right)  dtd\xi d\theta\right) 				,
		\end{align}
		\begin{align}
			\Pi_{i,1}\left(  \theta\right)   & =\frac{1}{2}\int_{0}^{\infty}K^{T}\left(
			t\right)  M_{i,1}\tilde K\left(  t,\theta\right)  dt 
			+\frac{1}{2}\int_{0}^{\infty}K^{T}\left(  t\right)  M_{i,1}\hat  K%
			(t,\theta)dt															 
			+\frac{1}{2}\int_{-h}^{0}\int_{0}^{\infty}K^{T}\left(  t\right)  M_{i,2}%
			\left(  \xi\right)  \tilde K\left(  t+\xi,\theta\right)
			dtd\xi													\label{Pi_1}\\
			& +\frac{1}{2}\int_{-h}^{0}\int_{0}^{\infty}K^{T}\left(  t\right)  M_{i,2}%
			\left(  \xi\right)  \hat  K(t+\xi,\theta)dtd\xi			 
			+\frac{1}{2}\int_{-h}^{0}\int_{0}^{\infty}K^{T}\left(  t+\xi\right)
			M_{i,2}^{T}\left(  \xi\right)  \tilde K\left(  t,\theta\right)
			dtd\xi																	  
			+\frac{1}{2}\int_{-h}^{0}\int_{0}^{\infty}K^{T}\left(  t+\xi\right)
			M_{i,2}^{T}\left(  \xi\right)  \hat  K\left(  t,\theta\right)
			dtd\xi													\nonumber\\
			& +\frac{1}{2}\int_{-h}^{0}\int_{-h}^{0}\int_{0}^{\infty}K^{T}\left(
			t+\xi\right)  M_{i,3}\left(  \xi,\theta\right)
			\tilde K\left(  t+\theta,\theta\right)  dtd\xi d\theta				
			+\frac{1}{2}\int_{-h}^{0}\int_{-h}^{0}\int_{0}^{\infty}K^{T}\left(
			t+\xi\right)  M_{i,3}\left(  \xi,\theta\right)
			\hat  K(t+\theta,\theta)dtd\xi d\theta,								\nonumber
		\end{align}
		
		\begin{align}
			\Pi_{i,2}\left(  \xi,\theta\right)   & =+\int_{-h}^{0}\int_{0}^{\infty}\hat  K^{T}(t,\xi)M_{i,2}\left(  \theta
			_{1}\right)  \hat  K(t+\xi,\theta)dtd\xi								
			+\frac{1}{2}\int_{-h}^{0}\int_{-h}^{0}\int_{0}^{\infty}\hat  K^{T}%
			(t+\xi,\xi)M_{i,3}\left(  \xi,\theta\right)
			\hat  K(t+\theta,\theta)dtd\xi d\theta					\label{Pi_2}\\
			& +\int_{0}^{\infty}\tilde K^{T}\left(
			t,\xi\right)  M_{i,1}\hat  K(t,\theta)dt 											 
			+\int_{-h}^{0}\int_{0}^{\infty}\tilde K^{T}\left(  t,\xi\right)  M_{i,2}%
			\left(  \xi\right)  \tilde K\left(  t+\xi,\theta\right)
			dtd\xi
			+\frac{1}{2}\int_{-h}^{0}\int_{-h}^{0}\int_{0}^{\infty}\tilde K^{T}\left(
			t+\xi,\xi\right)  M_{i,3}\left(  \xi,\theta\right)
			\tilde K\left(  t+\theta,\theta\right)  dtd\xi d\theta	\nonumber\\
			& +\frac{1}{2}\int_{0}^{\infty}\tilde K^{T}\left(  t,\xi\right)  M_{i,1}%
			\tilde K\left(  t,\theta\right)  dt												
			+\int_{-h}^{0}\int_{0}^{\infty}\tilde K^{T}\left(  t,\xi\right)  M_{i,2}%
			\left(  \xi\right)  \hat  K(t+\xi,\theta)dtd\xi	
			+\frac{1}{2}\int_{-h}^{0}\int_{-h}^{0}\int_{0}^{\infty}\tilde K^{T}\left(
			t+\xi,\xi\right)  M_{i,3}\left(  \xi,\theta\right)
			\hat  K(t+\theta,\theta)dtd\xi d\theta					\nonumber\\
			& +\frac{1}{2}\int_{0}^{\infty}\hat  K^{T}(t,\xi)M_{i,1}\hat  K(t,\theta)dt			
			+\int_{-h}^{0}\int_{0}^{\infty}\hat  K^{T}(t,\xi)M_{i,2}\left(  \theta
			_{1}\right)  \tilde K\left(  t+\xi,\theta\right)  dtd\xi
			+\frac{1}{2}\int_{-h}^{0}\int_{-h}^{0}\int_{0}^{\infty}\hat  K^{T}%
			(t+\xi,\xi)M_{i,3}\left(  \xi,\theta\right)  \tilde K\left(
			t+\theta,\theta\right)  dtd\xi d\theta	,\nonumber 
		\end{align}

		\normalsize	
		In (\ref{Pi_0}) we can se that it has a quadratic form and by the Schur complement \cite{haynsworth1970applications}, this matrix is positive definite if
		\begin{align} 
			M_{i,1} &>0, \label{positive cond1} \\
			M_{i,3}(\xi,\theta)-M^T_{i,2}(\xi)M_{i,1}M_{i,2}(\theta)&>0,\label{positive cond2}
		\end{align}
		also all the integrals are continuous and the symmetry on the argument of (\ref{Pi_2}) can be demonstrated by its transpose. 
		As (\ref{J}) corresponds to (\ref{V_p}) then the proposition has been demonstrated.
	\end{proof}

	Finally, Step \ref{paso3} of the method is stated in the following theorem:
	
	\begin{theorem}								\label{teorema_nash}
		Consider the system (\ref{S1}) and the performance indices (\ref{J_i}); then, the Nash strategy of the form
		  
		\begin{align}
			u_{i}^*\left(  x_t\right)  =&-\left(  R_{i,i}\right)  ^{-1}B_i^{T}\Pi_{i,0}x\left(  t\right)  \label{uop}\\
            -&\left(
			R_{i,i}\right)  ^{-1}B_i^{T}\int_{-h}^{0}\Pi_{i,1}\left(  \theta\right)
			x\left(  t+\theta\right)  d\theta;\nonumber					
		\end{align}\normalsize
		is admissible and minimizes $J_{i}(\varphi,u_i,u_{-i})$, if there exists a set of matrices $\Pi_{i,0}$, $\Pi_{i,1}(\theta)$ and $\Pi_{i,2}(\xi,\theta)$ with the properties given in the Proposition \ref{prop1}, that satisfies the following coupled partial differential equations:

		\begin{align}
			-\Phi_{i,0}   =&\Pi_{i,0}A_{0}+A_{0}^{T}\Pi_{i,0}-2\Pi_{i,0}  C_i\Pi_{i,0} -2\Pi_{i,0} C_{-i}\Pi_{-i,0}							\label{cond1}\\
			&  +  \Pi_{i,0}D_{i,i}\Pi_{i,0}   +   \Pi_{-i,0}D_{i,-i}\Pi_{-i,0}   +2\Pi_{i,1}	\left(  0\right) 						,\nonumber \\
			\frac{d\Pi_{i,1}\left(  \theta\right)  }{d\theta} =  &  A_{0}^{T}\Pi_{i,1}\left(  \theta\right)  +\Phi_{i,1}\left(  \theta\right)
			+\Pi_{i,2}\left(  0,\theta\right)						\label{cond2}\\
			&  + \Pi_{i,0}
			D_{i,i}\Pi_{i,1}\left(\theta\right)  + \Pi_{-i,0}D_{i,-i}\Pi_{-i,1}\left(  \theta\right) 											\nonumber\\
			&  -2\Pi_{i,0} C_{i}\Pi_{i,1}\left(  \theta\right)  		 				  -\Pi_{i,0}  C_{-i}\Pi_{-i,1}\left(  \theta\right)  		\nonumber \\
			& 	-  \Pi_{-i,0}C_{-i}  \Pi_{i,1}\left(  \theta\right) 		,\nonumber\\
			\frac{\partial\Pi_{i,2}\left(  \xi,\theta\right)  }{\partial \xi}+ & \frac
			{\partial\Pi_{i,2}\left(  \xi,\theta\right)  }{\partial \theta}=\Phi_{i,2}\left(\xi,\theta\right)				\label{cond3}\\
			& +   \Pi_{i,1}^{T}\left(  \xi\right) D_{i,i} \Pi_{i,1}\left(\theta\right) 	
			+   \Pi_{-i,1}^{T}\left(  \xi\right)  D_{i,-i} \Pi_{-i,1}\left( 	\theta\right)    \nonumber\\
			&  -2 \Pi_{i,1}^{T}\left(  \xi\right) C_{i} \Pi_{i,1}\left(  \theta\right)    
			-2 \Pi_{-i,1}^{T}\left(  \xi\right) C_{-i} \Pi_{i,1}\left(  \theta\right)    ,\nonumber\\
			A_{1}^{T}\Pi_{i,1}\left(  \theta\right)   = & \Pi_{i,2}\left(  -h,\theta\right)  ,\label{cond4}\\
			\Pi_{i,0}A_{1} = &  \Pi_{i,1}\left(  -h\right)  ,			\label{cond5}
		\end{align}\normalsize
		where

		\begin{align}
			C_{j}=& B_{j}\left(  R_{j,j}\right)  ^{-1}B_{j}^T   \label{def_c}\\
			D_{i,j}=&  B_{j}\left(  R_{j,j}\right)  ^{-1}	R_{i,j}\left(  R_{j,j}\right)  ^{-1}B_{j}^T ,\quad j=i,-i.  \label{def_d}
		\end{align}\normalsize
	\end{theorem}

	\begin{proof}
		The initial conditions of a time-delay system are $x_0 = \varphi$. Then, for $t \geq 0$, one can write $x_{0+t} = x_t$; this allows expressing (\ref{J}) as
		
		\begin{align*}
			V_{i}\left(  x_{t}\right)  =&x^{T}\left(  t\right)  \Pi_{i,0}x\left(
			t\right)  +2x^{T}\left(  t\right)  \int_{-h}^{0}\Pi_{i,1}\left(
			\theta\right)  x\left(  t+\theta\right)  d\theta			\\
			&+\int_{-h}^{0}\int_{-h}%
			^{0}x^{T}\left(  t+\xi\right)  \Pi_{i,2}\left(  \xi,\theta\right)  x\left(
			t+\theta\right)  d\theta d\xi,\nonumber		
		\end{align*}\normalsize
		$x_t$ is a system (\ref{S1}) trajectory, whose Hamiltonian, for each player, is 
		\begin{equation}			\label{H_h}
			H_{i}(x_{t},u_{i},u_{-i}) \triangleq \dot{V}_{i}\left(  x_{t}\right)\Bigg|_{_{\substack{(\ref{S1})
						\\u_{i}-admisible}}}+g_{i}(x_{t},u_{i},u_{-i}),
		\end{equation}  \normalsize
		where $g_{i}(x_{t},u_{i},u_{-i})$ is the argument from (\ref{J_i}) and 
		\begin{align*} 
			\dot{V}_i(x_t) &\big|_{(\ref{S1}),\,u_i\text{-admissible}} =\   
			x^T(t)\left(\Pi_{i,0}A_0 + A_0^T\Pi_{i,0}\right)x(t)+ 2x^T(t)\Pi_{i,0}A_1 x(t-h) \\
			& + 2x^T(t)\Pi_{i,0}\left(B_1 u_1(t)+B_2 u_2(t)\right) + 2x^T(t)\int_{-h}^{0} A_0^T\Pi_{i,1}(\theta) x(t+\theta) d\theta \nonumber\\
			& + 2x^T(t-h)A_1^T \int_{-h}^{0} \Pi_{i,1}(\theta)x(t+\theta) d\theta 
			  + 2x^T(t)\int_{-h}^{0} \Pi_{i,1}(\theta) \frac{\partial x(t+\theta)}{\partial t} d\theta \nonumber\\
			& + 2\int_{-h}^{0} \left(B_1 u_1(t)+B_2 u_2(t)\right)^T \Pi_{i,1}(\theta)x(t+\theta) d\theta \nonumber\\
			& + \int_{-h}^{0}\int_{-h}^{0} \frac{\partial x^T(t+\xi)}{\partial \xi} \Pi_{i,2}(\xi, \theta)x(t+\theta) d\theta d\xi \nonumber\\
			& + \int_{-h}^{0}\int_{-h}^{0} x^T(t+\xi) \Pi_{i,2}(\xi, \theta) \frac{\partial x(t+\theta)}{\partial \theta} d\theta d\xi ,  \nonumber
		\end{align*} \normalsize
        
		then 		
		\begin{align} \label{H_i}
			H_{i}&(x_{t},u_{i},u_{-i})   =\   
			x^T(t)\left(\Pi_{i,0}A_0 + A_0^T\Pi_{i,0}\right)x(t) \\
            & +2x^T(t)\Pi_{i,0}A_1 x(t-h)  + 2x^T(t)\Pi_{i,0}\left(B_1 u_1(t)+B_2 u_2(t)\right) \nonumber\\
			&+ 2x^T(t)\int_{-h}^{0} A_0^T\Pi_{i,1}(\theta) x(t+\theta) d\theta \nonumber
			  + 2x^T(t-h)A_1^T \int_{-h}^{0} \Pi_{i,1}(\theta)x(t+\theta) d\theta\nonumber\\
			& + 2x^T(t)\int_{-h}^{0} \Pi_{i,1}(\theta) \frac{\partial x(t+\theta)}{\partial t} d\theta \nonumber
			  + 2\int_{-h}^{0} \left(B_1 u_1(t)+B_2 u_2(t)\right)^T \Pi_{i,1}(\theta)x(t+\theta) d\theta \nonumber\\
			& + \int_{-h}^{0}\int_{-h}^{0} \frac{\partial x^T(t+\xi)}{\partial \xi} \Pi_{i,2}(\xi, \theta)x(t+\theta) d\theta d\xi \nonumber
			  + \int_{-h}^{0}\int_{-h}^{0} x^T(t+\xi) \Pi_{i,2}(\xi, \theta) \frac{\partial x(t+\theta)}{\partial \theta} d\theta d\xi   \nonumber\\
			& + x^{T}\left(
			t\right)  \Phi_{i,0}x\left(  t\right)  +2x^{T}\left(  t\right)  \int
			_{-h}^{0}\Phi_{i,1}\left(  \theta \right)  x\left(  t+\theta\right)  d\theta 													\nonumber\\
			& +  \int_{-h} ^{0}\int_{-h}^{0}x^{T}\left(  t+\xi\right)  \Phi_{i,2}%
			\left(  \xi,\theta\right)  x\left(  t+\theta\right)  d\theta d\xi				\nonumber	\\
			& + u_1^{T}(t)  R_{i,1} u_{1}\left(  t\right) + u_2^{T}\left(
			t\right)  R_{i,2} u_{2}\left(  t\right) dt		,  	\nonumber
		\end{align} \normalsize
		when the calculus of variations theorem, \cite{kirk2004optimal,bacsar1998dynamic}, is applied to find the minimum, 
		\begin{align*}
			\min_{u-admisible} H_{i}(x_{t},u_{i},u_{-i})=   H_{i}^*(x^*_{t},u_i^{*},u_
			{-i}^{*}) \\  
            \implies \dfrac{\partial H_{i}(x_{t},u_{i},u_{-i})}{\partial u_{i}}=0,
		\end{align*}\normalsize
		it results that 
		\begin{align*}
			\frac{\partial}{\partial u_{i}} & H_{i}(x_{t},u_{i},u_{-i})  = 2B_i^{T}\Pi_{i,0}x\left(  t\right) \\
            &+2B_i^{T}\int_{-h}^{0}\Pi_{i,1}
			\left(  \theta\right)  x\left(  t+\theta\right)  d\theta+2R_{i,i}%
			u_{i}\left(  t\right)  ,
		\end{align*} \normalsize
		has to be equal to zero, that is
		
		\begin{align*}
			2B_i^{T}\Pi_{i,0}x\left(  t\right)  +2B_i^{T}\int_{-h}^{0}\Pi_{i,1}\left(
			\theta\right)  x\left(  t+\theta\right)  d\theta+2R_{i,i}u_{i}\left(  t\right)=0,
		\end{align*}\normalsize
		solving to obtain the optimal strategy $u_{i}^*(x_t)$ 
        
		\begin{align*}
			u_i^{\ast}\left( x_t\right) &=-\left(  R_{i,i}\right)  ^{-1}B_i^{T}\Pi_{i,0}x\left(  t\right)  \\
            &-\left(
			R_{i,i}\right)  ^{-1}B_i^{T}\int_{-h}^{0}\Pi_{i,1}\left(  \theta\right)
			x\left(  t+\theta\right)  d\theta;													\nonumber			    
		\end{align*}\normalsize
		to verify that (\ref{uop}) is a minimum, the second derivative is computed, which must be positive definite. Since $R_{i,i} > 0$, a minimum has been found; therefore, (\ref{uop}) is optimal. However, to ensure that this result yields a global minimum, the Hamilton–Jacobi–Bellman equation must be satisfied, specifically when $t = 0,\ x_0 = \varphi$:
		
		\begin{equation*}
			\dot{V}_{i}(\varphi)\Bigg| _{\substack{(\ref{S1}) \\u_{i}=u_i^{\ast}
			}}+g(\varphi,u_i^*,u_{-i}^*)  =0				,				\label{HJB}
		\end{equation*}\normalsize
		then the control strategies (\ref{uop}) are substituted, taking into account the definitions (\ref{def_c}) and (\ref{def_d}) with $j=1,2$, in (\ref{H_i}) and the result is set equal to zero:
		
		\begin{align*}
			0 = &  \varphi^{T}\left(  0\right)  \left\{  \Pi_{i,0}A_{0}+A_{0}^{T}\Pi_{i,0}
			+\Phi_{i,0}-2\Pi_{i,0} \left( C_1\Pi_{1,0} + C_2\Pi_{2,0} \right) \right.  	\\
			&\quad  \left.  +  \Pi_{1,0}D_{i,1}\Pi_{1,0} + \Pi_{2,0}D_{i,2}\Pi_{2,0}  +2\Pi_{i,1}
			\left(  0\right)  \right\}  \varphi\left(  0\right) \\
			&  +2\varphi^{T}\left(  0\right)  \int_{-h}^{0}\left\{ 	\Pi_{1,0}D_{i,1}\Pi_{1,1}\left(  \theta\right) + \Pi_{2,0}D_{i,2}\Pi_{2,1}\left(  \theta\right)  \right. \\
			&\quad  +  A_{0}^{T}\Pi_{i,1}\left(  \theta\right)  +\Phi_{i,1}\left(
			\theta\right)  -\frac{d\Pi_{i,1}\left(  \theta\right)  }{d\theta}+\Pi_{i,2}
			\left(  0,\theta\right)	\\
			& \quad  -\Pi_{i,0} \left( C_1 \Pi_{1,1}\left(  \theta\right) + C_2 \Pi_{2,1}\left(  \theta\right)  \right)\\
            & \quad \left.  -\left( \Pi_{1,0}C_1 + \Pi_{2,0}C_2 \right)  \Pi_{i,1}\left( \theta\right)  \right\}  \varphi\left(  \theta\right) d\theta	\\
			&  +\int_{-h}^{0}\int_{-h}^{0}\varphi^{T}\left(  \xi\right)  \left\{  \Phi_{i,2}\left(  \xi,\theta\right) \right. \\
            & \quad+  \Pi_{1,1}^{T}\left(
			\xi\right) D_{i,1}\Pi_{1,1}\left(  \theta\right) + \Pi_{2,1}^{T}\left(
			\xi\right) D_{i,2}\Pi_{2,1}\left(  \theta\right)  		\nonumber\\
			& \quad  -2 \left(  \Pi_{1,1}^{T}\left(  \xi\right) C_1 + \Pi_{2,1}^{T}\left(  \xi\right) C_2 \right)  \Pi_{i,1}\left(  \theta\right) \\
            & \quad\left. -\frac
			{\partial\Pi_{i,2}\left(  \xi,\theta\right)  }{\partial \xi}-\frac{\partial\Pi_{i,2}
				\left(  \xi,\theta\right)  }{\partial \theta}\right\}  \varphi\left(  \theta\right)	d\xi d\theta								\nonumber\\
			&  +2\varphi^{T}\left(  -h\right)  \int_{-h}^{0}\left\{  A_{1}^{T}\Pi_{i,1}
			\left(  \theta\right)  -\Pi_{i,2}\left(  -h,\theta\right)  \right\}  \varphi\left(
			\theta\right)  d\theta											\nonumber\\
			&  +2\varphi^{T}\left(  0\right)  \left\{  \Pi_{i,0}A_{1}-\Pi_{i,1}\left(
			-h\right)  \right\}  \varphi\left( -h\right) 					,\nonumber
		\end{align*}\normalsize
		since this must hold for any $\varphi$, the only possibility is that the expression inside the braces is zero, or that conditions (\ref{cond1})--(\ref{cond5}) of Theorem \ref{teorema_nash} are satisfied; therefore, the theorem is proven.
	\end{proof}
	Please note that the i-smooth calculus \cite{kim2015smooth} was not used here. Instead, the classical constructive procedure proposed in \cite{krasovskii1962analytic} for obtaining the time derivative of functionals in time-delay systems is employed.
	The equations resulting from Theorem \ref{teorema_nash} form a system of five equations per player, that is, a total of ten simultaneous equations, whose analytical solution is not straightforward to obtain. For this reason, in this proposal, we use the approach presented in \cite{kim2000linear} to decouple the set of equations \eqref{cond1}--\eqref{cond5} and obtain explicit formulas for the closed-loop Nash strategies. This is given in the following Theorem.
	
	\begin{theorem} \label{solution}
		Let the system be (\ref{S1}), if the matrices from (\ref{J_i}) are 
		\begin{align}
			\Phi_{i,0} &  =  N_{i}+2\Pi_{i,0}C_{-i}\Pi_{-i,0}-\Pi_{-i,0}D_{i,-i}\Pi
			_{-i,0} \label{phi_0}\\
            & -2\Pi_{i,1}\left(  0\right)  ,\nonumber\\
			\Phi_{i,1}\left(  \theta\right)  &  = -\Pi_{i,2}\left(  0,\theta\right)
			-\Pi_{-i,0}D_{i,-i}\Pi_{-i,1}\left(  \theta\right)   \label{phi_1}\\
			& +\Pi_{i,0}C_{-i}\Pi
			_{-i,1}\left(  \theta\right) +\Pi_{-i,0}C_{-i}\Pi_{i,1}\left(  \theta\right)
			,\nonumber\\
			\Phi_{i,2}\left(  \xi,\theta\right)  &  = -\Pi_{i,1}^{T}\left(  \xi\right)
			\left[  D_{i,i}-2C_{i}\right]  \Pi_{i,1}\left(  \theta\right)   \label{phi_2}\\
			& -\Pi_{-i,1}%
			^{T}\left(  \xi\right)  D_{i,-i}\Pi_{-i,1}\left(  \theta\right) +2\Pi
			_{-i,1}^{T}\left(  \xi\right)  C_{-i}\Pi_{i,1}\left(  \theta\right) , \nonumber
		\end{align}\normalsize
		where $N_{i}$ are positive definite parametric free matrices in $\mathbb{R}^{n\times n}$, then $\Pi_{i,0}$ is solution to the Algebraic Riccati Equation (ARE) 
		\begin{align}\label{solare}
			\Pi_{i,0}A_{0}+A_{0}^{T}\Pi_{i,0}+\Pi_{i,0}\left( D_{i,i}-2C_{i}\right)  \Pi_{i,0}=-N_{i},
		\end{align}	\normalsize
		and the (\ref{cond2})--(\ref{cond5}) equations have solution as follows 
		\begin{align} \label{solode}
			\Pi_{i,1}(\theta)=e^{F_i(\theta+h)}\,A_1\Pi_{i,0} ,\qquad \theta\in[-h,0],
		\end{align}
		\begin{align}\label{solpde}
			\Pi_{i,2}(\xi,\theta)=
			\begin{cases}
				\Psi_{i}(\xi)\,\Pi_{i,1}(\theta), & (\xi,\theta)\in\Omega_1,\\[4pt]
				\Pi_{i,1}^T(\xi)\,\Psi_i^T(\theta), & (\xi,\theta)\in\Omega_2,
			\end{cases}
		\end{align}\normalsize
		where  $F_i:= A_{0}^{T}+\Pi 
		_{i,0}[D_{i,i}-2C_{i}]$, $\Psi_i(\xi)=A_1^{T}
			e^{- F_i(\xi+h)}$, $\Psi(-h)=A_1^T$ y 	
		$\Omega_1=\{(\xi,\theta)\in[-h,0]\times[-h,0]:\, \xi-\theta<0\},\
		\Omega_2=\{(\xi,\theta)\in[-h,0]\times[-h,0]:\, \xi-\theta>0\},$ moreover, the closed loop Nash strategies of (\ref{S1}) can be constructed as in (\ref{uop}). 	
	\end{theorem}

	\begin{proof}
		If the  terms in (\ref{cond1})–(\ref{cond3}) are chosen as (\ref{phi_0})–(\ref{phi_2}), then, upon substitution, one obtains 
		\begin{align}
			\Pi_{i,0}A_{0}+A_{0}^{T}\Pi_{i,0}+\Pi_{i,0}\left( D_{i,i}-2C_{i}\right)  \Pi_{i,0}	    &=-N_{i}				,		\label{algric}\\
			\left[  A_{0}^{T}+\Pi
			_{i,0}(D_{i,i}-2C_{i})\right]  \Pi_{i,1}\left(  \theta\right) & = \frac{d\Pi_{i,1}\left(  \theta\right)  }{d\theta}, \label{oderic}\\
			\frac{\partial\Pi_{i,2}\left(  v,\theta\right)  }{\partial v}+\frac{\partial
				\Pi_{i,2}\left(  v,\theta\right)  }{\partial \theta}  &  =0		,	\label{pderic}\\
			A_{1}^{T}\Pi_{i,1}\left(  \theta\right)   &  =\Pi_{i,2}\left(  -h,\theta\right), \label{frontric1}\\
			\Pi_{i,0}A_{1}  &  =\Pi_{i,1}\left(  -h\right)		.	\label{frontric2}\
		\end{align}\normalsize
		
		This provides a direct solution system for each player, since (\ref{algric}) is an ARE; whenever the pair $[A, D_{i,i} - 2C_i]$ is stabilizable, $\Pi_{i,0}$ exists \cite{kalman1960contributions}. Now, if one substitutes this solution $\Pi_{i,0}$ into the equation given by \eqref{oderic}, this equation is a first-order matrix differential equation whose solution is known, just as in (\ref{solode}), with initial condition given by (\ref{frontric2}).

		Let us verify that $\Pi_{i,2}(\xi,\theta)$ satisfies the transport equation 
		\begin{align}\label{ec_transp}
			\frac{\partial \Pi_{i,2}(\xi,\theta)}{\partial \xi} +
			\frac{\partial \Pi_{i,2}(\xi,\theta)}{\partial \theta}=0.
		\end{align}\normalsize
		
		Suppose that in the domain  $\Omega_1$ the following holds:
		$\Pi_{i,2}(\xi,\theta)=\Psi_i(\xi)\Pi_{i,1}(\theta)$,
		upon substitution on (\ref{ec_transp}) it results 
		\begin{align*}
			\frac{d\Psi_i(\xi)}{d\xi}\Pi_{i,1}(\theta)
			+
			\Psi_i(\xi)\frac{d\Pi_{i,1}(\theta)}{d\theta}=0,
		\end{align*}\normalsize
		substituting (\ref{oderic}) into the last equation:
				
		\begin{align*}
			\frac{d\Psi_i(\xi)}{d\xi}\Pi_{i,1}(\theta)
			+
			\Psi_i(\xi)
			F_i \Pi_{i,1}\left(  \theta\right)=0,
		\end{align*}\normalsize
		As $\Pi_{i,1}(\theta)\neq \textbf{0}$, then 
		
		\begin{align*}
			\frac{d\Psi_i(\xi)}{d\xi}
			+\Psi_i(\xi) F_i 	=0,	
		\end{align*}\normalsize
		this is a first order matrix differential equation whose solution on $[-h,0]$ is
				
		\begin{align*}
			\Psi_i(\xi,\Psi_i(-h))
			=\Psi_i(-h)e^{-F_i(\xi+h)}.
		\end{align*}\normalsize
		
		To obtain the initial conditions of this system, lets use (\ref{frontric1}); this yields that $\Psi_i(-h)=A_1^{T}$, therefore
		
		\begin{align} \label{dem3}
			\Psi_i(\xi)=A_1^{T}
			e^{-F_i(\xi+h)},
		\end{align}\normalsize
		consequently, in the domain $\Omega_1$ it holds that $ \Pi_{i,2}(\xi,\theta)=\Psi_i(\xi)\Pi_{i,1}(\theta).
		$
		
		Now, for the domain $\Omega_2$, the same procedure is followed, which results in the transpose of (\ref{dem3}); therefore, the proposed form for  $\Pi_{i,2}(\xi,\theta)$ satisfies the differential equation. Finally, the symmetry property follows directly by taking the transpose of (\ref{pderic}) \cite{kim2000linear}. This completes the proof of the theorem.
	\end{proof}

    This result can be expressed as a type of a classical Nash equilibrium strategy as in Theorem \ref{Nash non delay}, however, do not forget that this is for a state delay system, described in the next Corollary.

	\begin{corollary}
		If (\ref{phi_0}) is chosen as 
		\begin{align}
			\Phi_{i,0}   = N_{i} -2\Pi_{i,1}\left(  0\right)  ,\label{phi_0J}
		\end{align}\normalsize
		and let (\ref{phi_1}), (\ref{phi_2}) as they are, 
		then $\Pi_{i,0}$ is the solution of 
		\begin{align}\label{solnash}
			-N_i  =&\Pi_{i,0}(A_{0}-C_{-i}\Pi_{-i,0})+(A_{0}-C_{-i}\Pi_{-i,0})^{T}\Pi_{i,0} \\
            &+  \Pi_{i,0}(D_{i,i}-2C_i)\Pi_{i,0}+   \Pi_{-i,0}D_{i,-i}\Pi_{-i,0}   ,	\nonumber	 
		\end{align}	\normalsize
		which results in the classical solution of a Nash equilibrium, and the solutions (\ref{solode}) and (\ref{solpde}) abide.
	\end{corollary}

	The closed loop stability of system (\ref{S1}), is not guaranteed with the Bellman equation, since it is not obvious to demonstrate that the Bellman functional is  non-negative, due to the cross terms; therefore, the closed-loop stability system, (\ref{S1_lc}), must be verified to conclude that the Nash strategies are admissible. To this end, let us follow the stability procedure presented  in Theorem 7.2.1 of Kolamanovskii exposed in \cite{kolmanovskii2012applied}.
	
	Consider the change of variable $s = t - \theta$ in the closed loop system given by (\ref{S1_lc}), so that
	
	\begin{align}
		\dot{x}\left(  t\right)  =\tilde{A}_{0}x\left(  t\right)  +A_{1}x\left(
		t-h\right)  +\int_{t-h}^{t}G\left(  s-t\right)  x\left(  s\right)
		ds,																		\label{S1_lc_s}
	\end{align}\normalsize
	Once the closed-loop system has been rewritten in this form, sufficient conditions for its stability are established in the following theorem.
	\begin{theorem}\label{estability}
		Let 
		\begin{align} \label{V_est}
			V(x_t)=x^T(t)Px(t)+\int_{t-h}^{t}x^T(s)Rx(s)\\
            +\int_0^h\int_{t-s}^tx^T
			(v)Sx(v)dvds,\nonumber
		\end{align}\normalsize
		with $P,R,S > 0$ and symmetric matrices. $V(\cdot)$ is a mapping \(\mathbb{R}^n \to \mathbb{R}\) Lyapunov-Krasovskii functional candidate for the closed-loop system (\ref{S1_lc_s}) and let $\alpha_1(w),\ \alpha_2(w):\mathbb{R}^+\to\mathbb{R}^+$, continuous nondecreasing functions, positive for $w > 0$, $\alpha_1(0)=\alpha_2(0)=0$, be 
		\begin{align}
			\alpha_1(\|x(t)\|)= & x^T(t)Px(t), \label{lowerbound}\\
			\alpha_2(\|x_t\|) = & \| x_t \|_h^2(\lambda_{max}(P)+h\lambda_{max}(R)+h^2\lambda_{max}(S)), \label{upperbound}
		\end{align}  \normalsize
		such that 
		\begin{align*}
			\alpha_1(\|x(t)\|) \le  V(x_t) \le \alpha_2(\|x_t\|), \quad \forall x \in \mathbb{R}^n,
		\end{align*}\normalsize
		then \(V(x_t)\) is positive definite and radially unbounded in the functional space $PC([-h,0],\mathbb{R}^n) $, even more, if there exists solution to the ARE 
		\begin{align} \label{ARE_est}
			-Q=&\tilde{A}_{0}^{T}P+P\tilde{A}_{0}+Sh+R\\
            &+P\left[  A_{1}R^{-1}A_{1}^{T}+\int
			_{-h}^{0}G^{T}\left(  v\right)  S^{-1}G\left(  v\right)  dv\right]  P\nonumber
		\end{align}\normalsize
		then the origin of the system is asymptotically stable as per Theorem \ref{Stability}, which implies that the Nash strategies (\ref{uop}) are stabilizing.
	\end{theorem} 
	
	\begin{proof}
		The bounds for the functional $V(.)$ can be calculated as (\ref{lowerbound})--(\ref{upperbound}) by direct calculations and appropriate majorizations;  so there it only is required to be shown that the derivative of the functional (\ref{V_est}) is nonnegative; therefore, it must be differentiated using Leibniz’s rule for integrals:  
		\begin{align*}
			\dot{V}(\cdot)=&\dot{x}^T(t)Px(t)+{x}^T(t)P\dot x(t)\\
			&+{x}^T(t)Rx(t)-{x}^T(t-h)Rx(t-h)+{x}^T(t)Shx(t)\nonumber\\
            &-\int_0^h{x}^T(t-s_1)Sx(t-s_1)ds_1	,	\nonumber
		\end{align*}\normalsize
		by performing a change of variable $s=t-s_1$ in the integral term,
		and also substituting (\ref{S1_lc_s}), it results
		
		\begin{align*}
			\dot{V} (\cdot) & =x^{T}\left(  t\right)  \left[  \tilde{A}_{0}^{T}P+P\tilde{A}%
			_{0}+Sh+R+PA_{1}R^{-1}A_{1}^{T}P\right.  \\
			& \left.  +P\int_{t-h}^{t}G^{T}\left(  s-t\right)  S^{-1}G\left(  s-t\right)
			dsP\right]  x(t)  \nonumber\\
			& -\left[  A_{1}^{T}Px(t)-Rx(t-h)\right]  ^{T}R^{-1}\left[  A_{1}%
			^{T}Px(t)-Rx(t-h)\right]  \nonumber\\
			& -\int_{t-h}^{t}\left[  G\left(  s-t\right)  Px\left(  t\right)  -S{x}%
			^{T}(s)\right]  ^{T}S^{-1}\left[  G\left(  s-t\right)  Px\left(  t\right)
			-S{x}^{T}(s)\right]  ds,\nonumber
		\end{align*}\normalsize
		with the variable change $v = s - t$, it is direct to see that:
        
		\begin{align*}
			\dot{V}(\cdot)\leq &x^{T}\left(  t\right)  \left[  \tilde{A}_{0}^{T}
			P+P\tilde{A}_{0}+Sh+R+PA_{1}R^{-1}A_{1}^{T}P	\right.		\\
			&+\left. P\int_{-h}^{0}G^{T}\left(  v\right)  S^{-1}G\left(  v\right)  dvP\right]  x(t).			\nonumber
		\end{align*}\normalsize
		Since the last term integrand is known, it can be computed and results in a constant matrix, it follows that it suffices for the ARE (\ref{ARE_est}) to admit a solution $P > 0$, with $P = P^{T}$, for $Q > 0$. Therefore, the theorem has been proven.
	\end{proof}
    \section{Experimental results}\label{apli}
	To validate the approach proposed here, a prototype was built that simulates a food dehydration system. It regulates temperature using two controllers/players (i.e., two heating elements) and a state feedback that accounts for the delayed term; this is a scalar system. The prototype and its identification via Recursive Least Squares (RLS) are described.
    The system can be operated using industrial REX–C100 Proportional-Integral-Derivative (PID) controllers or with an ESP32 and an AC dimmer module. The components and  data acquisition protocol  are specified. Fig. \ref{fig:diseno}. provides a general illustration of the system in the form of an instrumentation diagram; Fig. \ref{fig:chasis}. shows the elements that are used for this prototype.
    \begin{figure}[h]
    \centering    \includegraphics[width=0.7\linewidth]{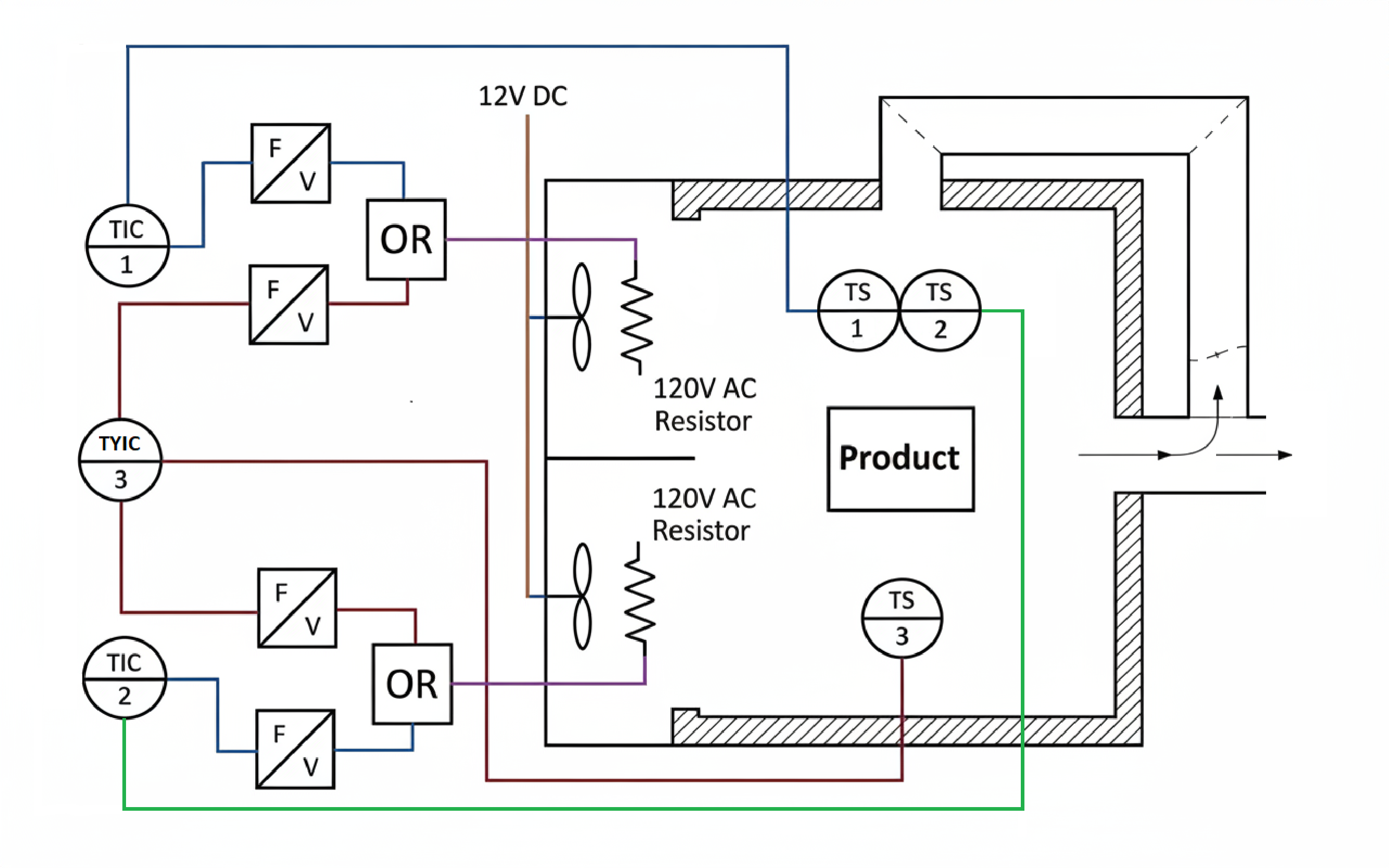}
    \caption{Prototype instrumentation diagram.}
    \label{fig:diseno}
    \end{figure}

    \begin{figure}[h]
    \centering    \includegraphics[width=0.7\linewidth]{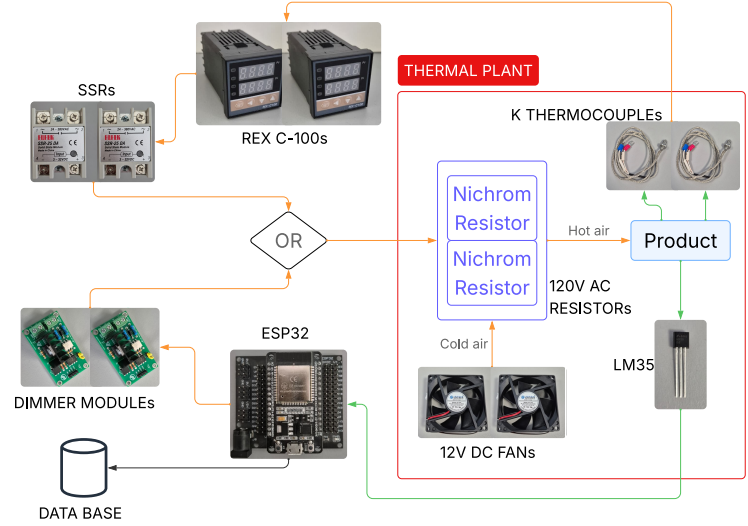}
    \caption{Elements and functionality diagram. The orange arrow carries two signals, the green arrow carries one signal, and the black arrow carries four signals.}
    \label{fig:chasis}
    \end{figure}
    One way to operate the prototype is based on classical control using two commercial PIDs with Solid State Relay (SSR), where the strategy signals $u_1, u_2$ are applied to the plant as a duty cycle in a PWM signal.  The duty cycle of a burst window computed by the REX–C100 industrial PIDs . Another way to control the plant is by using an ESP32 system with two AC dimmer modules, where the control signals are computed via the phase angle, and the set-point is expressed as an angle \(\alpha \in [10^\circ, 170^\circ]\). The module maps can be defined as: 
	\begin{align*}
		\alpha_\text{deg} =
		\begin{cases}
			170 - 160\,\dfrac{t_H - 1\text{ ms}}{1\text{ ms}}, & \text{VP, } t_H\in[1,2]\text{ ms}\\[6pt]
			170 - 160\,\dfrac{V_A}{5\text{ V}}, & \text{VA, } V_A\in[0,5]\text{ V}
		\end{cases}
	\end{align*}\normalsize
	And $\alpha = \alpha_\text{deg}\frac{\pi}{180}.$ In this case, VP is used, which computes the turn-on time for the dimmer’s PWM. Table~\ref{componentes2} describes the components used in this prototype.

 \begin{table}[h!]\centering \caption{Main components of the prototype.} \begin{tabular}{@{}ll@{}} \toprule Element & Specifications \\ \midrule Actuators & Nichrom Resistor 120 V AC, 200 W (\(\approx 1.67~\mathrm{A}\) each) \\ SSR & SSR-40DA, Out AC 24–380 V, In 3–32 VDC \\ AC Dimmer module & Angle control; BTA24-800; 5 A. \\ PID controllers & REX-C100 (SSR output) \\ ESP32 & GPIO \\ Thermal sensors & K thermocouple; LM35 (10 mV/°C) \\ \bottomrule \end{tabular}\label{componentes2} \end{table}
 From Fig. \ref{fig:diseno}. a thermal state is considered where the state $x(t)$ (control loop 3 on the Figure) represents the temperature inside the product chamber, measured in . The inputs $u_1, u_2$  represent the hot air produced by each resistor/player, and \(h > 0\) is the transport delay associated with the air re-injection to the plant by the pipe, i.e. $x(t-h)$  (see in Fig. \ref{fig:diseno}. the feedback pipe) as a scalar model of the  form given by model (\ref{S1}). An identification is required to obtain the parameters \(a_0, a_1, b_1,\) and \(b_2\) of this system. With sampling time \(T = 0.5~\mathrm{s}\) and $h = 2~\mathrm{s}$, then we have \(n_h = h/T = \mathbf{4}\) samples. The magnitude of delay $h$ is experimentally determined.
	
	RLS was applied to the stored data, with a forgetting factor \(\lambda = 0.999\), in \emph{off-line} mode. The comparison of measurement-approximation and persistent excitation signals applied to the plant is shown in Fig. \ref{fig:ident-rls}.
    \begin{figure}[h]
\centering
\includegraphics[width=0.8\linewidth]{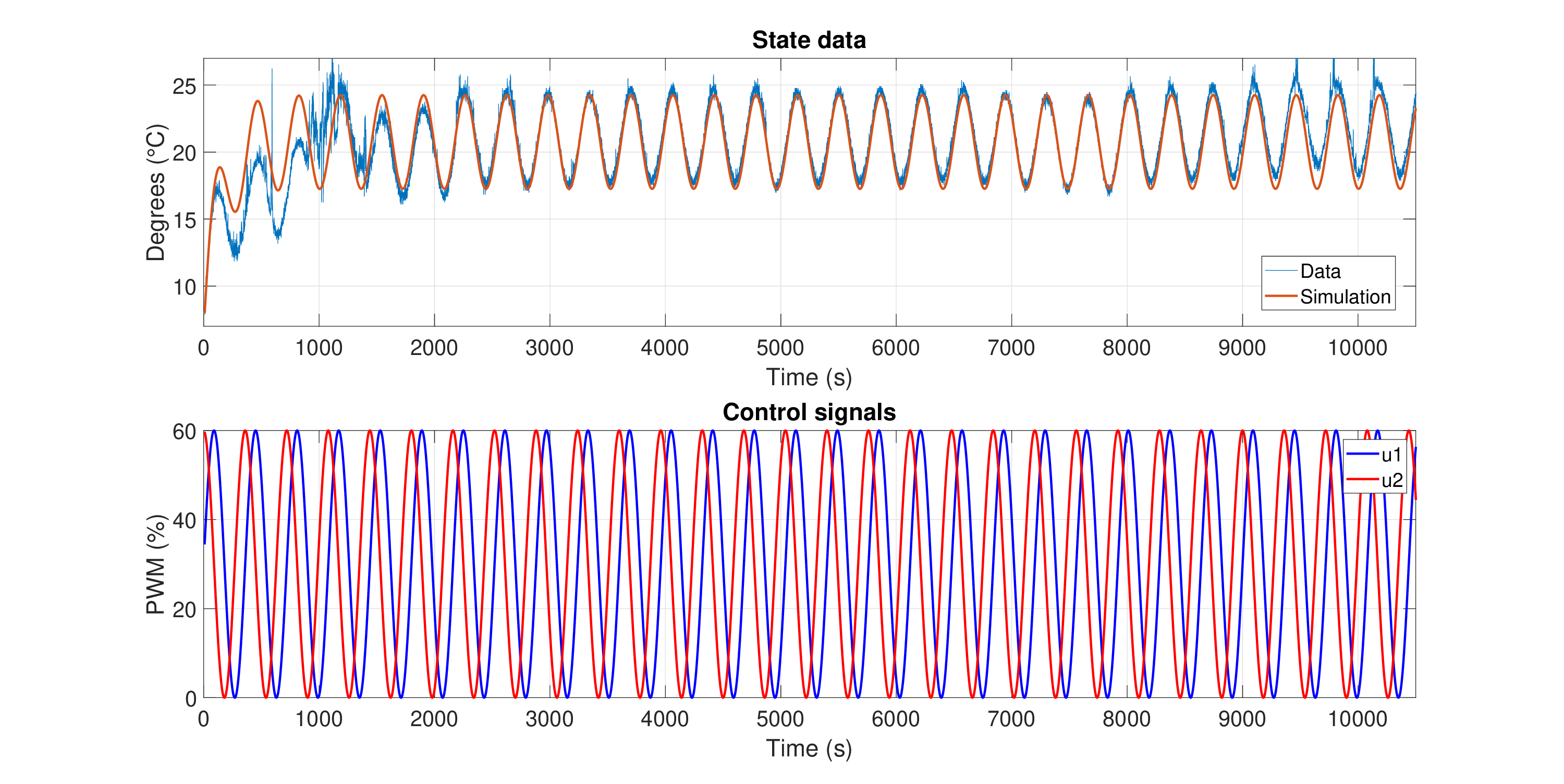}
\caption{RLS-based system identification and persistent excitation signals.}
\label{fig:ident-rls}
\end{figure}
The estimated parameters are: $
		a_0 = -0.730330839081807 ,\
		a_1 = 0.712505286229602 , \
		b_1 = 0.022853934400486 ,\
		b_2 = 0.028647298011767,$ for the model given by (\ref{S1}).
	Now, it is possible to apply the control algorithm  (the required matrices are computed in MATLAB). For this, choose: $M_1 = 500.0,\ M_2 = 50.0,\ R_{1,1} = 2.0,\   R_{1,2} = 0.2,\
	R_{2,1} = 0.1\  R_{2,2} = 1.0;$ and the calculated values for $\Pi_{0,i}$ and $\Phi_{0,i}$ are:
	$
		\Pi_{0,1}=366.2998,	\
		\Pi_{0,2}=34.9160,	\
		\Phi_{0,1}=420.7441,\
		\Phi_{0,2}=44.0238,\ $
	it can be seen that \(\Pi_{0,i}\) and \(\Phi_{0,i}\) are positive. The functions: $\Pi_{1,i}$ and $\Phi_{1,i}$, are shown in Fig. \ref{phi1fn}. where, the $*$ marks the points used of the functions $\Pi_{1,i}$ to compute the control strategies, because the data is acquired every 0.5 seconds. In Fig. \ref{phi2fn}. we can see the square regions given by the functions $\Pi_{2,i}$ and $\Phi_{2,i}$, and their  evolution.  Using these values, it is possible to apply the control strategies, which are approximated from (\ref{uop}) as:
    
		\begin{align}
			u_{i}^*\left(  k\right)  =&-\left(  R_{i,i}\right)  ^{-1}B_i^{T}\Pi_{i,0}x\left(  k\right) \\
            -&\left(
			R_{i,i}\right)  ^{-1}B_i^{T}\sum_{k_1=k-n_h}^{n_h}\Pi_{i,1}\left(  k_1\right)
			x\left(  k+k_1\right)  T.\nonumber				
		\end{align}\normalsize

        \begin{figure}[h]
        \centering		\includegraphics[width=0.8\linewidth]{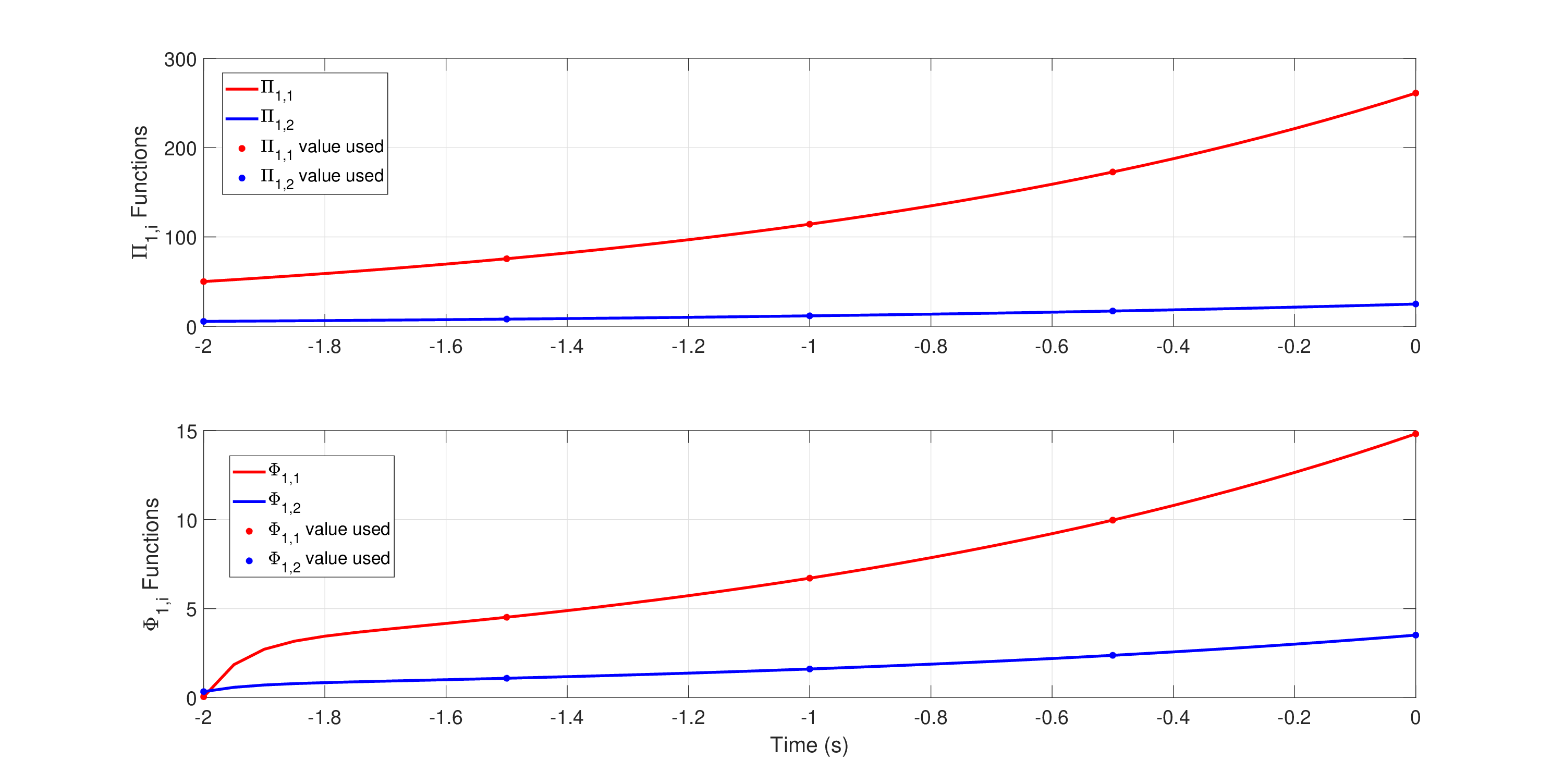}
		\caption{$\Pi_{1,i}$ and $\Phi_{1,i}$ functions.}
		\label{phi1fn}
	\end{figure}
    \begin{figure}[h]\centering	\includegraphics[width=0.8\linewidth]{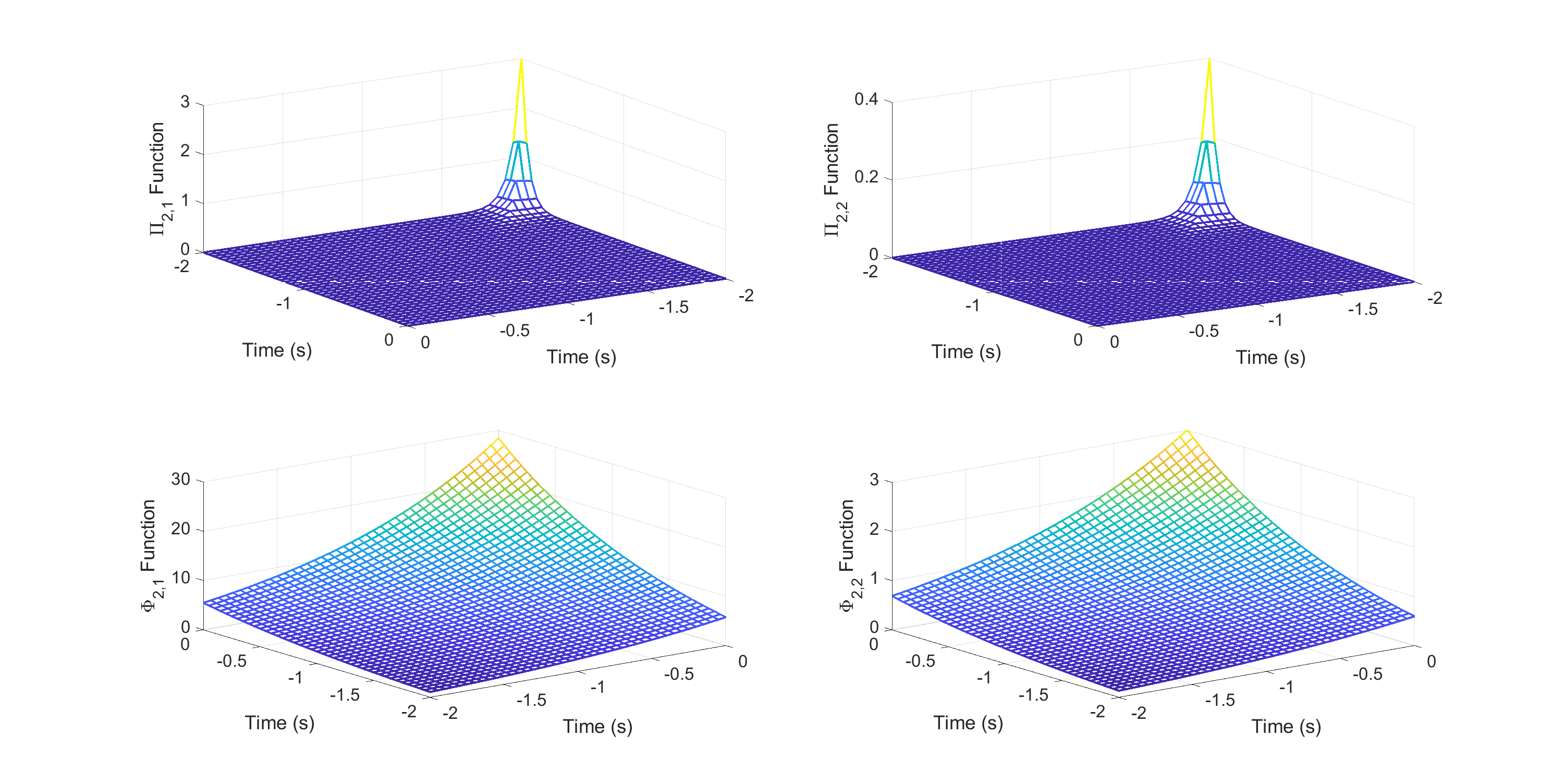}
		\caption{$\Pi_{2,i}$ and $\Phi_{2,i}$ functions.}
		\label{phi2fn}
	\end{figure}
    To validate the feasibility of this control strategy, a comparison with the classical control was carried out. Two PI on the REX-C100 controllers. Assuming the system is linear, it can be described as a sum of the effects of both controllers, due to the superposition property. These controllers are computed using the optimal algorithm in \cite{he2000pi} using a decoupled performance index. Identification is performed using the step response and MATLAB’s `tfest` command, which provides the first order transfer functions: 
	\begin{align*}
		\dfrac{G(s)}{U_1(s)}=\dfrac{0.0005756}{s+0.0001538}, \qquad \dfrac{G(s)}{U_2(s)}=\dfrac{0.0004266}{s+0.0001158}.
	\end{align*}\normalsize

	The controllers are required to stabilize the plant at 2000 seconds with a damping factor \(\zeta = 0.99\). The gains obtained by the algorithm are shown in Table \ref{gananciasPI}.
	
	\begin{table}[h!]\centering 
		\caption{PI controller parameters.}
		\begin{tabular}{c c c c c} 
			\toprule
			Player & $K_p$ & Proportional Band & $K_i$ & $T_i$ \\ \midrule
			$u_1$ & 6.642 & 15.054 & 0.000689 & 963.254 \\
			$u_2$ & 9.051 & 11.047 & 0.009304 & 972.839 \\ \bottomrule
		\end{tabular}\label{gananciasPI}
	\end{table}

	The results of the comparison between the controllers are shown below. Fig. \ref{comp}. presents the state and error comparison, and Fig. \ref{controlsignals}. shows the control signals applied by the Nash strategy. It is not possible to display the control signals applied by the PI controllers, since the signal sent by the REX-C100 would be altered by adding a sensor; and Table \ref{errors} shows the results of the error indices.
    \begin{figure}[h]\centering	\includegraphics[width=0.8\linewidth]{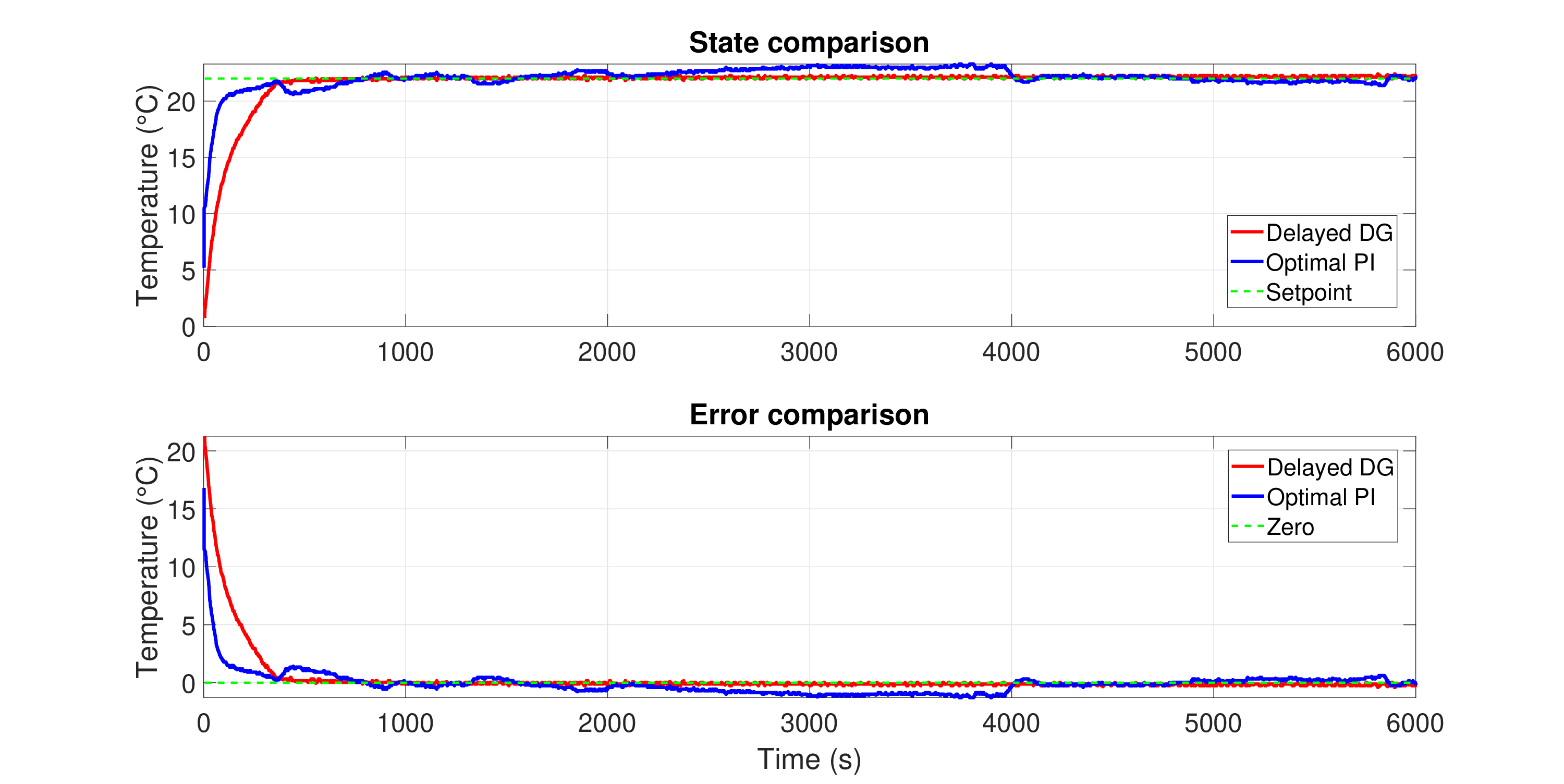}
		\caption{System behavior with PI controllers and the Nash strategy.}
		\label{comp}
	\end{figure}
	
	\begin{figure}[h]\centering	\includegraphics[width=0.8\linewidth]{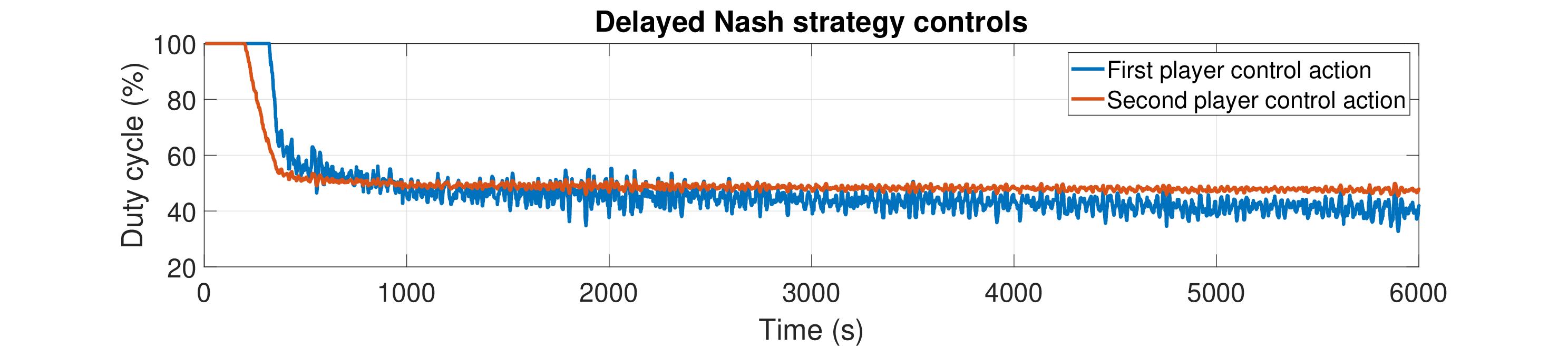}
		\caption{Control signals applied by the Nash strategy.}
		\label{controlsignals}
	\end{figure}
	
		 \begin{table}[h!]\centering 
		\caption{Comparison error indices.}
		\begin{tabular}{c c c}
			\toprule
			Index & Optimal PI & Nash strategy \\ \midrule
			ISE & $6.6886\times10^{3}$ & $2.4011\times10^{4}$\\
			IAE & $52.0637$ & $46.2959$\\
			ITSE & $1.5752\times10^{7}$ & $4.8313\times10^{6}$\\
			ITAE & $3.3610\times10^{5}$ & $2.6910\times10^{5}$\\
			\bottomrule
		\end{tabular}\label{errors}
	\end{table}

    The performance index (\ref{J_i}) can be expressed as %

    \begin{align*}
    J_{i}\left(  x_{t},u_{i},u_{-i}\right)    & =\int_{0}^{\infty}\left\{
    \int_{-h}^{0}\int_{-h}^{0}%
    \begin{bmatrix}
    x\left(  t\right)  \\
    x(t+\xi)
    \end{bmatrix}
    ^{T}%
    \begin{bmatrix}
    \Phi_{i,0} & \Phi_{i,1}\left(  \theta\right)  \\
    \Phi_{i,1}^{T}\left(  \theta\right)   & \Phi_{i,2}\left(  \xi,\theta\right)
    \end{bmatrix}%
    \begin{bmatrix}
    x\left(  t\right)  \\
    x(t+\theta)
    \end{bmatrix}
    d\xi d\theta\right.  \\
    & \left.  +u_{1}^{T}\left(  t\right)  R_{i,1}u_{1}\left(  t\right)  +u_{2}%
    ^{T}\left(  t\right)  R_{i,2}u_{2}\left(  t\right)  \right\}  dt,
    \end{align*}
    this is a cuadratic form, and if the matrix
    
    \[
    \Phi_{i}(h)=%
    \begin{bmatrix}
    \Phi_{i,0} & \Phi_{i,1}\left(  h\right)  \\
    \Phi_{i,1}^{T}\left(  h\right)   & \Phi_{i,2}\left(  h,h\right)
    \end{bmatrix}
    \]
    is positive definite, then the performance index is positive, as this matrix
    is simmetric, then their eigenvalues need to be positive. In $\left[  -h,0\right]  $
    we can calculate the eigenvalues of this matrix when the functions are used
    whith a MATLAB calculation, wich results in
    
    \begin{align*}
    eig\left(  \Phi_{1}(-2.0)\right)    & =\left[  1.0962,420.7441\right],  \\
    eig\left(  \Phi_{1}(-1.5)\right)    & =\left[  2.4465,420.7686\right] , \\
    eig\left(  \Phi_{1}(-1.0)\right)    & =\left[  5.4808,420.8347\right]  ,\\
    eig\left(  \Phi_{1}(-0.5)\right)    & =\left[  12.3337,420.9752\right]  ,\\
    eig\left(  \Phi_{1}(0.0)\right)    & =\left[  27.7937,421.2943\right],
    \end{align*}
    for the first player and
    
    \begin{align*}
    eig\left(  \Phi_{2}(-2.0)\right)    & =\left[  0.1315,44.0264\right] , \\
    eig\left(  \Phi_{2}(-1.5)\right)    & =\left[  0.2734,44.0424\right] , \\
    eig\left(  \Phi_{2}(-1.0)\right)    & =\left[  0.5825,44.0771\right] , \\
    eig\left(  \Phi_{2}(-0.5)\right)    & =\left[  1.2560,44.1514\right] , \\
    eig\left(  \Phi_{2}(0.0)\right)    & =\left[  2.7145,44.3190\right],
    \end{align*}
    for the second player, here, one can see that for this case the performance index is positive and also the conditions can be transported from (\ref{positive cond1})--(\ref{positive cond2}) to this verification.
    
    Additionally, to verify the closed loop system stability with the Nash strategies, one can use the Theorem \ref{estability} and a MATLAB computation, if we use the functions $\Pi_{0,i},\ \Pi_{1,i},$ and choose $R = 1,\ S = 0.1,\ Q = 0.1;$, then the integrand is 0.0989 and the solution of the (\ref{ARE_est}) ARE is $P=1.0899>0$ therefore the system is stable.
    
    \section*{Appendix: Quadratic performance index}
	When the performance index is quadratic, that is,
    \begin{align*}
		J_{i}\left(  u_{i},u_{-i}\right)    =\int\left(  x^{T}\left(
		t\right)  Q_i x\left(  t\right)  + \sum_{j=1}^2 u_j^{T}(t)  R_{i,j} u_{j}\left(  t\right) \right)
		dt,			
	\end{align*}\normalsize
	the procedure to be followed is the same as that outlined in Section
	 \ref{Resultados teoricos}.
	
	A particular form of the strategies $u_i^{*}(x_t)$ is selected, as given by equation (\ref{u_i}). The loop is then closed, and the solution is obtained on the Cauchy form (\ref{Sol_Cauchy}). With this, the performance index of each player is expressed, which turns out to be exactly as in (\ref{J_m}), but the expressions (\ref{Ms1})--(\ref{Ms3}) change to 
	\begin{align*}
		M_{i,1} &= Q_i+ \sum_{j=1}^2
		\Gamma_{j,0}^{T}R_{i,j}\Gamma_{j,0}, \\
		M_{i,2}\left(  \theta\right) &= \sum_{j=1}^{2}
		\Gamma_{j,0}^{T}R_{i,j}\Gamma_{j,1}\left(  \theta\right), \\
		M_{i,3}\left(  \xi,\theta\right)  &= \sum_{j=1}^2
		\Gamma_{j,1}^{T}\left(  \xi\right)  R_{i,j}\Gamma_{j,1}\left(
		\theta\right). 	
		\end{align*}
		
	When carrying out the minimization, equations (\ref{Pi_0})--(\ref{Pi_2}) retain the same form; the only difference is that in the Hamiltonian (\ref{H_h}), the term 
	$$g_i(\cdot)= x^{T}\left(
	t\right)  Q_i x\left(  t\right)  + \sum_{j=1}^2 u_j^{T}(t)  R_{i,j} u_{j}\left(  t\right),$$ even so, the form of the strategies remains as in (\ref{uop}); only the conditions that must be satisfied for the optimization of system (\ref{cond1})--(\ref{cond5}) change to
	
	\begin{align}
		-Q_{i}   =&	\Pi_{i,0}A_{0}+A_{0}^{T}\Pi_{i,0}+\sum_{j=1}^{2}
		\Pi_{j,0} D_{i,j} \Pi_{j,0} \\
		&+2\Pi_{i,1}\left(  0\right)  -2\Pi_{i,0}\sum_{j=1}^{2}C_{j}\Pi_{j,0}  ,\nonumber\\
		\frac{d\Pi_{i,1}\left(  \theta\right)  }{d\theta} =&	\Pi_{i,2}\left( 0,\theta\right)  -\Pi_{i,0} \sum_{j=1}^{2}
		C_{j}\Pi_{j,1}\left(  \theta\right)
		+  A_{0}^{T}\Pi_{i,1}\left(  \theta\right) 	 \\ 
		& + \sum_{j=1}^{2}	\Pi_{j,0}D_{i,j}\Pi_{j,1}\left(  \theta\right) -  \sum_{j=1}^{2}\Pi_{j,0}C_{j} \Pi_{i,1}\left(	\theta\right)									,\nonumber\\
		\frac{\partial\Pi_{i,2}\left(\xi,\theta\right)}{\partial\theta}&
		+\frac{\partial\Pi_{i,2}\left( \xi,\theta\right)}{\partial\xi}
		=\sum_{j=1}^{2}
		\Pi_{j,1}^{T}\left(  \xi\right)  C_{j} \Pi_{j,1}\left(\theta\right)  \\
        &-2 \sum_{j=1}^{2}
		\Pi_{j,1}^{T}\left(  \xi\right)  C_{j}  \Pi_{i,1}\left(  \theta\right)   			,\nonumber\\
		A_{1}^{T}\Pi_{i,1}\left(  \theta\right)    &=\Pi_{i,2}\left(
		-h,\theta\right)									,\\
		\Pi_{i,0}A_{1}   &=\Pi_{i,1}\left(  -h\right)		.
	\end{align}
	
	Indeed, due to the couplings in the equations, this system does not have a simple analytical solution, to the best of the authors’ knowledge.

	\begin{remark}
		By solving this problem, the results can be extended to $N$-players system, by changing the sum limits from $2$ to $n$.
	\end{remark}
	
    
\bibliographystyle{unsrt}  

\bibliography{references} 





\end{document}